\documentclass[12pt]{article}

\usepackage{amsmath,amsfonts,amssymb,amsthm}
\usepackage[non-compressed-cites, non-sorted-cites]{amsrefs}
\usepackage{hyperref}

\usepackage{enumitem}

\newenvironment{enumlist}{\begin{enumerate}[label=\upshape{(\arabic*)}]}{\end{enumerate}}

\title{Using Ramsey's Theorem Once}

\author{
Jeffry L. Hirst\thanks{Department of Mathematical Sciences, Appalachian State University,
Walker Hall, Boone, NC 28608. \textit{Email:}~\texttt{hirstjl@appstate.edu}}%
 \and %
Carl Mummert\thanks{Department of Mathematics, Marshall University, 1 John Marshall Drive, Huntington, WV 25755. \textit{Email:}~\texttt{mummertc@marshall.edu}}%
 }
 
\date{June 1, 2017}

\theoremstyle{plain}
\newtheorem{theorem}{Theorem}

\newtheorem{corollary}[theorem]{Corollary}  
\newtheorem{lemma}[theorem]{Lemma}

\theoremstyle{definition}
\newtheorem{defn}[theorem]{Definition}      

\newcommand{\nat}{\mathbb N}  
\newcommand{\rca}{{\sf RCA}_0}
\newcommand{\rcaw}{{\sf RCA}_0^\omega}

\newcommand{\gth}{\mathcal{T}}
\newcommand{\haw}{{\widehat{\mathsf{E}\text{-}\mathsf{HA}}}\strut^\omega_{\raise2pt\hbox{\scriptsize${\mathord{\upharpoonright}}$}}}
\newcommand{\qfac}{{\sf QF}\text{-}{\sf AC}^{1,0}}
\newcommand{\rt}{{\sf{RT}}}
\newcommand{\probP}{{\sf P}{\rm :}\forall x (p_1 (x) \to \exists y \, p_2 (x,y))}

\newcommand{\probQu}{{\sf Q}{\rm :}\forall u (q_1 (u) \to \exists v \, q_2 (u,v))}
\newcommand{\wlt}{\le_{W}}

\begin{document}

\maketitle

\begin{abstract}
We show that $\rt (2,4)$ cannot be proved with one typical application of $\rt (2,2)$ in an intuitionistic
extension of $\rca$ to higher types, but that this does not remain true when the law of the excluded
middle is added.  The argument uses Kohlenbach's axiomatization of higher
order reverse mathematics, results related to modified reducibility, and a formalization of Weihrauch
reducibility.

\textbf{MSC Subject Class (2000):} 03B30; 03F35; 03F50; 03D30; 03F60
 \end{abstract}
 
 \vskip -5in
 \noindent
 This is a preprint of an article that appeared in the Archive for Mathematical Logic, DOI: 10.1007/s00153-019-00664-z
 
 \noindent
 An author approved manuscript incorporating the referee's suggested changes, and a copy of the corrigendum on the last page can be found on the authors' websites.
 \vskip 4.75in
 
 One of the questions motivating the exploration of uniform reductions in the article
 of Dorais, Dzhafarov, Hirst, Mileti, and Shafer~\cite{doraisetal} was:
 Is it possible to prove Ramsey's theorem for pairs and four colors from a single use of
 Ramsey's theorem for pairs and two colors?  Not surprisingly, the answer depends
 on the base system chosen, as shown in \S\ref{section3} below.  Our approach utilizes a
 formalization of Weihrauch reducibility described by Hirst at Dagstuhl Seminar 15392~\cite{dagstuhl}, based on higher order reverse mathematics as axiomatized by
 Kohlenbach~\cite{k2001}.
 { This choice of formalization, along with the choice of different base systems, yields results that differ from
 those in recent closely related work of Kuyper~\cite{rk}.} { We discuss these differences at the end of Section~\ref{section2}.}

\section{Formal Weihrauch reduction}\label{section1}

The counting of theorem applications in later sections relies in part on the close
connection between proofs in some systems of arithmetic and Weihrauch reduction.
This relationship is also central to the arguments of Kuyper~\cite{rk}
Rather than formalizing Weihrauch reduction by means of indices (as in~\cite{rk}),
we work in extensions of reverse mathematics axiom systems \cite{simpson} to higher types,
first formulated by Kohlenbach~\cite{k2001}.  These systems have variables
for numbers (type $0$ objects), functions from numbers to numbers (type $1$ objects
encoding sets of numbers), and for functions from type $1$ functions to numbers,
type $1$ functions to type $1$ functions, and so on.  In Kohlenbach's terminology,
$\rcaw$ consists of $\haw$ plus the law of the excluded middle and $\qfac$, a restricted
choice scheme.  The system $\haw$ is an axiomatization of intuitionistic Heyting arithmetic
in all finite types, with restricted induction and primitive recursion.  For full details, see Kohlenbach~\cite{kbook}*{\S3.4}.  The choice scheme $\qfac$ asserts
\[
\forall x \exists n A(x,n) \to \exists \varphi \forall x A(x,\varphi(x))
\]
for $A$ quantifier free, where $x$ is a set variable, $n$ is a number variable, and $\varphi$ is a variable
for functions mapping sets to numbers.  To make the typography more compact, we will use letters
between $i$ and $n$ to denote number variables, letters following $s$ in the alphabet as set variables,
and greek letters for various functionals.  We also use $i\rcaw$ to denote the intuitionistic system
arising from omitting the law of the excluded middle from $\rcaw$.  A concise outline of the axioms
for $\rcaw$ can be found in the article of Hirst and Mummert~\cite{hirstmummert}.

Weihrauch reducibility is a computability theoretic approach to measuring relative uniform strength.
See Brattka and Gherardi~\cite{bg2011b} for an extensive survey.  We adopt the notion of
reduction of problems, as used by Dorais~\cite{dorais} and Dorais {\it et al.}~\cite{doraisetal}.  A problem
$\sf{P}$ is a formula of the form $\forall x (p_1 (x) \to \exists y \, p_2 (x,y))$, asserting that whenever
$x$ is an instance of the problem then there is a solution $y$ for $x$.  Suppose
$\probP$ and $\probQu$ are problems.  We say ${\sf Q}$ is {\em Weihrauch reducible} to ${\sf P}$, and write ${\sf Q} \wlt {\sf P}$,
if there are computable functions $\varphi$ and $\psi$ such that the following hold:
\begin{list}{$\bullet$}{}
\item  If $u$ is an instance of $\sf Q$ then $\varphi (u)$ is an instance of $\sf P$, that is:
\[q_1 (u) \to p_1 (\varphi (u))\]
\item  If $y$ is a solution of $\varphi(u )$, then $\psi (u , y )$ is a solution of $\sf Q$, that is:
\[p_2 ( \varphi (u)  , y )\to q_2 ( u , \psi (u , y ))\]
\end{list}
Consequently, in the language of $\rcaw$ we can formalize ${\sf Q}\wlt {\sf P}$ as:
\[
\exists \varphi \exists \psi 
\forall u \left( q_1 (u) \to \left( p_1(\varphi (u)) \land
\forall y [ p_2 (\varphi (u) , y) \to q_2 (u, \psi (u,y))]
\right)\right)
\]
We will be working in subsystems of higher order reverse mathematics, so we use ${\sf Q} \wlt {\sf P}$ as an
abbreviation for the formula above, despite the fact that the leading quantifiers in the formula
are not explicitly restricted to computable functionals.  When working in $i\rcaw$, for many choices of $\sf Q$ and $\sf P$ this
is a faithful translation, as shown by Corollary~\ref{4F}.  However, in the classical setting, $\rcaw \vdash {\sf Q} \wlt {\sf P}$
may not imply ${\sf Q} \wlt {\sf P}$, as shown by the example following Corollary~\ref{4F}.

\section{Counting theorem applications}\label{section2}

In this section, we show that formalized Weihrauch reducibility is closely related to the structure of some intuitionistic proofs.

\begin{defn}\label{B1}
Suppose $\gth$ is a theory and $\probP$ and $\probQu$ are problems.
We say $\gth$ {\sl proves} $\sf Q$ {\sl with one typical use of} $\sf P$ if the following two sentences hold:
\begin{enumlist}
\item \label{B1A} For a variable $u$ there is a term $x_u$ such that using only axioms of $\gth$ and the assumption $q_1 (u)$,
and with no applications of generalization to $u$ or any variables appearing free in $x_u$, there is a deduction
of $p_1(x_u)$.
\item \label{B1B}  For a previously unused constant symbol $y_0$, there is a term $v_{x_u, y_0}$ such that using only axioms of $\gth$ and the
assumption $ p_2 (x_u,y_0)$, and with no applications of generalization to $u$ or any
variable appearing free in $x_u$ or $v_{x_u, y_0 }$, there is a deduction of $q_2 (u, v_{x_u , y_0} )$.
\end{enumlist}
\end{defn}


Informally, this definition says that given an instance of the problem $\sf{Q}$, there is an instance $x_u$ of the problem $\sf P$ such
that if there is a solution $y_0$ to $x_u$ then there is a solution $v_{x_u , y_0}$ to $\sf Q$.  The restrictions on generalization insure the
validity of applications of the deduction theorem in the proof of the following lemma.

\begin{lemma}\label{B2}
Suppose $\gth$ is a theory that includes intuitionistic predicate calculus, 
$\probP$ and $\probQu$ are problems, and $\gth$ proves $\sf Q$ with
one typical use of $\sf P$.  Then $\gth$ proves:
\[
\forall u \exists x \forall y \exists v (q_1(u)\to(p_1(x) \land (p_2(x,y) \to q_2(u,v))))
\]
\end{lemma}

\begin{proof}
Given a proof of $\sf Q$ in $\gth$ with one typical use of $\sf P$, build a new proof as follows.
Assume $q_1 (u)$ as a hypothesis and, applying sentence~\ref{B1A} of Definition~\ref{B1},
emulate the given proof to construct a term $x_u$ with $p_1(x_u )$.
Let $y_0$ be a new constant symbol and
assume $p_2(x_u,y_0)$ as a hypothesis.
By sentence~\ref{B1B} of Definition~\ref{B1}, we can find a term $v_{x_u,y_0}$ and prove $q_2(u,v_{x_u , y_0} )$.
One application of the deduction theorem
yields $p_2(x_u, y_0) \to q_2 (u,v_{x_u , y_0})$.
By $\land$-introduction~\cite{kleene}*{\S19, Ax. 3} followed by the deduction theorem, we have:
\[
q_1(u) \to (p_1(x_u) \land (p_2(x_u , y_0 ) \to q_2 (u , v_{x_u, y_0} ))).
\]
Note that $x_u$ depends only on $u$ and $v_{x_u, y_0}$ depends only on $y_0$ and $x_u$.  Alternating
applications of $\exists$-introduction~\cite{kleene}*{\S{32}, fla.~68} and $\forall$-introduction~\cite{kleene}*{\S{32}, fla.~64} yield
\[
\forall u \exists x \forall y \exists v (q_1(u)\to(p_1(x) \land (p_2(x,y) \to q_2(u,v))))
\]
as desired.
\end{proof}

A formula is {\em $\exists$-free} if it is built from prime (that is, atomic) formulas using only universal quantification
and the connectives $\land$ and $\to$.  Here, the symbol $\bot$ is considered prime, and $\neg A$ is an abbreviation
of $A \to \bot$, so $\exists$-free formulas may include both $\bot$ and $\neg$.  Troelstra's~\cite{troelstra} collection $\Gamma_1$ consists
of those formulas defined inductively by the following:
\begin{list}{$\bullet$}{}
\item  All prime formulas are elements of $\Gamma_1$.
\item  If $A$ and $B$ are in $\Gamma_1$, then so are $A \land B$, $A \lor B$,
$\forall x A$, and $\exists x A$.
\item  If $A$ is $\exists$-free and $B$ is in $\Gamma_1$ then $\exists x A \to B$ is in $\Gamma_1$,
where $\exists x$ may represent a block of existential quantifiers.
\end{list}

\begin{theorem}\label{B4}
Suppose $\probP$ and $\probQu$ are problems and
the formula
$q_1(u)\to(p_1(x) \land [p_2(x,y) \to q_2(u,v)])$, abbreviated as $R(x,y,u,v)$, is in $\Gamma_1$.
Then $i\rcaw\vdash \forall u \exists x \forall y \exists v R(x,y,u,v)$ if and only if
$i\rcaw \vdash {\sf Q} \wlt {\sf P}$.
\end{theorem}

\begin{proof}
To prove the implication from left to right,
suppose $\sf P$, $\sf Q$, and $R$ are as hypothesized, and $i\rcaw \vdash \forall u \exists x \forall y \exists v R(x,y,u,v)$.
The proof of Lemma~3.9 of Hirst and Mummert~\cite{hirstmummert} also holds for $i\rcaw$, so by two applications
of that lemma, there are terms $x_u$ and $v_{x_u , y}$ such that $i\rcaw \vdash \forall u \forall y R(x_u , y ,u , v_{x_u ,y} )$.
$i\rcaw$ proves existence of functionals $\varphi(u)=x_u$ and $\psi(u,y)=v_{x_u, y}$, so $i\rcaw$ proves:
\[
\forall u \forall y (q_1(u)\to(p_1(\varphi(u)) \land [p_2(\varphi(u),y) \to q_2(u,\psi(u,y))]))
\]
which is equivalent to ${\sf Q} \wlt {\sf P}$ by intuitionistic predicate calculus via~\cite{kleene}*{\S35, fla.~95},
\cite{kleene}*{\S35, fla.~89}, and $\exists$-introduction~\cite{kleene}*{\S32, fla.~68}.

Note that the proof of Lemma~3.9 of~\cite{hirstmummert} is based on versions of the soundness theorem for
modified realizability, which appears as Theorem~5.8 of Kohlenbach~\cite{kbook} and Theorem~3.4.5 of
Troelstra~\cite{troelstra}, and conversion lemmas for modified reducibility, Lemma~5.20 of Kohlenbach~\cite{kbook}
and Lemma~3.6.5 of~Troelstra~\cite{troelstra}.  The conversion lemmas are restricted to formulas in $\Gamma_1$,
necessitating the inclusion of this restriction as a hypothesis for this argument.

To prove the converse, suppose $i\rcaw\vdash {\sf Q} \wlt {\sf P}$.  
Thus, by our formalization adopted in section \S\ref{section1}, $i\rcaw$
proves the existence of functionals $\varphi$ and $\psi$ satisfying
\[
\forall u(q_1(u) \to
(p_1(\varphi (u))\land
\forall y [p_2(\varphi(u),y) \to q_2 (u, \psi (u, y))]
)
)
\]
By intuitionistic predicate calculus~\cite{kleene}*{\S{35}, fla.~89} and~\cite{kleene}*{\S35, fla.~95}, we can
move the universal quantifier on $y$ to the front of the formula.
Applying appropriate quantifier elimination followed by quantifier introduction yields
$\forall u \exists x \forall y \exists v R(x,y,u,v)$.
\end{proof}

As a corollary, we can show a close relationship between intuitionistic proofs and formal Weihrauch reducibility in
intuitionistic systems.

\begin{corollary}\label{B5}
Suppose $\sf P$, $\sf Q$ and $R$ satisfy the hypotheses in Theorem~\ref{B4}.
Then $i\rcaw$ proves $\sf Q$ with one typical use of $\sf P$ if and only if $i\rcaw \vdash {\sf Q} \wlt {\sf P}$.
\end{corollary}

\begin{proof}
The forward implication follows immediately from Lemma~\ref{B2} and Theorem~\ref{B4}.  To prove the converse,
suppose $i\rcaw$ proves the existence of functions $\varphi$ and $\psi$ witnessing ${\sf Q} \wlt {\sf P}$.
Then $\varphi (u)$ satisfies sentence~\ref{B1A} of Definition~\ref{B1}, so $p_1 (\varphi(u))$.
Assume the single use of $\sf P$ given by $p_2(\varphi(u) , y_0)$.
Because ${\sf Q} \wlt {\sf P}$, we have $q_2 (u, \psi (u, y_0 ))$, completing a proof satisfying sentence~\ref{B1B}
of Definition~\ref{B1}.
\end{proof}

Theorem~\ref{B4} also allows us to show that formal Weihrauch reducibility proved in $i\rcaw$ is often a faithful
representation of actual Weihrauch reducibility.

\begin{corollary}\label{4F}
Suppose $\sf P$, $\sf Q$ and $R$ satisfy the hypotheses in Theorem~\ref{B4}.
If $i\rcaw \vdash {\sf Q} \wlt {\sf P}$, then ${\sf Q} \wlt {\sf P}$.
\end{corollary}

\begin{proof}
For $\sf P$, $\sf Q$ and $R$ as hypothesized, if $i\rcaw \vdash {\sf Q} \wlt {\sf P}$ then by Theorem~\ref{B4},
$i\rcaw \vdash \forall u \exists x \forall y \exists v R (x,y,u,v)$.  As in the proof of Theorem~\ref{B4},
this means there are terms $x_u$ and $v_{x_u, y}$ in the language of $i\rcaw$ such that
$i\rcaw \vdash \forall u \forall y R(x_u ,y ,u , v_{x_u , y} )$.  Thus in any model of $i \rcaw$ based on
$\omega$ and the power set of $\omega$, where the basic arithmetic function symbols and the combinators
have their usual interpretations, the interpretations of the functionals $\lambda u . x_u$ and $\lambda (x_u, y) . v_{x_u , y}$
(that is, $\varphi$ and $\psi$ as in the proof of Theorem~\ref{B4}) will be computable functionals witnessing
${\sf Q} \wlt {\sf P}$.
\end{proof}

Corollary~\ref{4F} does not hold if $i\rcaw$ is replaced by $\rcaw$.  For example, suppose $\sf P$ is the trivial
problem defined by using $0=0$ for both $p_1$ and $p_2$.  Thus every set is an acceptable input for $\sf P$, and
every set is a solution of $\sf P$ for any input.  To define the problem $\sf Q$, let $T$ be an infinite computable
binary tree (all nodes labeled $0$ or $1$) with no infinite computable path.  Viewing an input $u$ as a function from
$\nat$ to $\nat$, we may interpret $u$ as a sequence of zeros and ones by identifying $u(n)$ with $0$ if
$u(n) = 0$ and identifying $u(n)$ with $1$ if $u(n)\neq 0$.  Let $q_1$ be $0=0$ so every input is acceptable for $\sf Q$.
Let $q_2 (u,v)$ say that either $v(0)=0$ and $p(n) = v(n+1)$ is an infinite path in $T$, or $v(0)>0$ and
$\langle u(0), \dots u(v(0))\rangle \notin T$.  Since $T$ is $\Delta^0_1$ definable, $q_2 (u,v)$ can be
written as a $\Pi^0_1$ formula.

Working in $\rcaw$, we will prove that $\exists \psi \forall u\, q_2(u, \psi (u))$.  By the law of the excluded middle,
either $T$ has an infinite path or it doesn't, so either $\exists p \forall n \, \langle p(0), \dots p(n) \rangle \in T$ or
$\forall p \exists n \, \langle p(0), \dots p(n) \rangle\notin T$.
In the first case, choose an infinite path $p_0$ and define $\psi$ to be the constant functional that maps each input
to the sequence $0$ followed by $p_0$.  In the second case, let $\psi$ map each $u$ to the function that
always takes the value $1+\mu m ( \langle u(0), \dots u(m) \rangle \notin T)$, so for each $u$ and $n$, $\psi (u) (n)$ is
a positive witness that $u$ is not an infinite path.  In either case, $\forall u \, q_2 (u, \psi (u))$, as desired.
Consequently, the identity functional $\varphi$ trivially witnesses
\[
\forall u (0=0 \to (0=0 \land \forall y (0=0 \to q_2 (u, \psi (u))))),
\]
so $\rcaw$ proves that ${\sf Q} \wlt {\sf P}$.

Turning to the computability theoretic framework, we will show that $\sf Q$ is not Weihrauch reducible to
to $\sf P$.  To see this, suppose by way of contradiction that $\varphi$ and $\psi$ are computable functionals
witnessing ${\sf Q} \wlt {\sf P}$.  Because $\sf P$ is trivial, $\emptyset$ (the constant $0$ function) is a solution of $\varphi (u)$,
so for all $u$, $\psi (u, \emptyset )$ is a solution of $\sf Q$.
By K\"onig's Lemma, let $u_0$ be an infinite path through $T$.
Because there is no witness that $u_0$ is not an infinite path,
we must have $\psi (u_0 , \emptyset ) (0) = 0$.  The functional $\psi$ is computable, so for some finite $k$, if $u$ is any
extension of $\langle u_0 (0) , \dots u_0 (k) \rangle$, then $\psi (u , \emptyset )(0)= 0$.
Choose a computable sequence $s_0$ such that $s_0$ extends $\langle u_0 (0) , \dots u_0 (k) \rangle$.  Then
$\psi (s_0, \emptyset)$ is a solution of $\sf Q$ and
$\psi (s_0, \emptyset ) (0) = 0$, so $v_0$ defined by $v_0 (n)= \psi(s_0 , \emptyset ) (n+1)$ is an infinite path
through $T$.  But $v_0$ is computable, contradicting the choice of $T$ and completing the example.

{
We close this section by comparing Corollary~\ref{B5} with Theorem~7.1 of Kuyper~\cite{rk}. 
The results are similar in that each states the equivalence of the existence of a formalized Weihrauch reduction
with the existence of a restricted resource proof of a related formula.}
{ Neither result implies the other, however. On one hand, Kuyper's results assume Markov's Principle in the base system, while ours do not. On the other hand,} 
{ the class of pairs of problems $\sf P$ and $\sf Q$
such that $i\rcaw$ proves $\sf Q$ with one typical use of $\sf P$ is a proper subclass of those for which
$({\sf EL}_0+\sf{MP})^{\exists \alpha a}$ proves $\sf P^\prime \to \sf Q^\prime$ (in the sense of Kuyper~\cite{rk}). This is an immediate consequence of Corollary~\ref{B5}, Theorem~7.1 of Kuyper~\cite{rk}, and the following theorem.

\begin{theorem}
Suppose $\sf P$, $\sf Q$ and $\sf R$ satisfy the hypotheses of Theorem~\ref{B4}.  If $i\rcaw \vdash {\sf Q}\wlt {\sf P}$, then there
are standard natural number indices $e_0$ and $e_1$ such that $\rca$ proves that $e_0$ and $e_1$ witness that
$\sf Q$ Weihrauch reduces to $\sf P$ as formalized in Theorem~7.1 of Kuyper~\cite{rk}.  The converse of this implication fails.
\end{theorem}

\begin{proof}
Suppose $\sf P$, $\sf Q$ and $\sf R$ satisfy the hypotheses of Theorem~\ref{B4}.  Note that the formalization of
${\sf Q} \wlt {\sf P}$ in $i\rcaw$ is in $\Gamma_1$.  By the intuitionistic analog of Lemma 3.9 of Hirst and Mummert \cite{hirstmummert},
there are terms in the language of $i\rcaw$ corresponding to the functionals witnessing ${\sf Q} \wlt {\sf P}$.  The desired
indices can be calculated from these terms.

To prove that the converse fails, let $\sf P$ be the trivial problem $\forall x (0=0 \to \exists y ~ (0=0))$
and let $\sf Q$ be the problem 
\[
\forall u (0=0 \to \exists v( \forall n ~u(n) =0 \lor \exists n ~ u(n) \neq 0)).
\]
Note that $\sf P$, $\sf Q$, and the associated formula $\sf R$ satisfy the hypotheses of Theorem~\ref{B4}.
In $i\rcaw$, ${\sf Q} \wlt {\sf P}$ implies $\forall u ( \forall n ~u(n) =0 \lor \exists n ~ u(n) \neq 0))$.  Because
this conclusion (a form of the Lesser Principle of Omniscience) is not intuitionistically valid, $i\rcaw$ does
not prove ${\sf Q} \wlt {\sf P}$.  On the other hand, for any indices $e_0$ and $e_1$, the classical system $\rca$ proves
\[
\forall u ( 0=0 \to ( 0=0 \to (0=0 \land \forall y (0=0 \to (\forall n~u(n)=0 \lor \exists n~u(n) \neq 0))))),
\]
so for any choice of indices, $\rca$ proves that $\sf Q$ Weihrauch reduces to $\sf P$ in the sense
of Theorem~7.1 of Kuyper~\cite{rk}.
\end{proof}

In light of the preceding example, it would be nice to know if this distinction between the formalizations of
Weihrauch reducibility holds for more combinatorially interesting choices of $\sf P$ and $\sf Q$.  That is,
can we find natural choices of $\sf P$ and $\sf Q$ such that $\sf Q$ is not a theorem of $\rca$, $\rca$ proves
$\sf Q$ assuming $\sf P$, $({\sf EL}_0 + {\sf MP})^{\exists \alpha a}$ proves ${\sf P}^\prime \to {\sf Q}^\prime$, and
$i\rcaw$ cannot prove $\sf Q$ with one typical use of $\sf P$?  }

{ Unlike our results,} { the results of Kuyper~\cite{rk} are not restricted
to formulas in $\Gamma_1$, in part due to his utilization of the Kuroda negative translation.  We wonder whether
similar methods can extend the results of this paper and our previous results \cite{hirstmummert} to a wider class
of formulas.}

\section{Ramsey's theorem}\label{section3}

We can use the preceding results to address our question about proofs of Ramsey's theorem.
Let $\rt (2,4)$ denote the following formulation of Ramsey's theorem for pairs and four colors:
If $f\colon [\nat]^2 \to 4$, then there is an infinite $x \subset \nat$ and an $i<4$ such that
$f([x]^2) = i$.  The set $x$ is called {\em monochromatic}.  Similarly, $\rt(2,2)$ denotes Ramsey's theorem
for pairs and two colors.

For any $k$, we can formalize $\rt (2,k)$ as a particularly simple $\Pi^1_2$ formula.  In the higher
order axiom systems described by Kohlenbach~\cite{k2001}, all higher order objects are functions,
with subsets of $\nat$ being encoded by characteristic functions or by enumerations.
Pairs of natural numbers can be encoded by a single natural number, so any function from $\nat$ into $\nat$
(that is, any type 1 object) can be viewed as a function from $[\nat]^2$ into $\nat$.  By composition with
a truncation function $t_n$ defined by $t_n(m)=m$ if $m<n$ and $t_n(m)=0$ otherwise,
we may view any type 1 function as a map from $[\nat]^2$ into $n$.  Using these notions,
we can formalize $\rt (2,4)$ as
\[
\forall f \exists x \forall m (x(m) < x(m^\prime) \land
\forall 0<i<j<m (t_4(f(x(i),x(j)))=x(0))).
\]
Formalized in this fashion, $\rt (2,4)$ is in $\Gamma_1$ and its matrix (the portion beginning with $\forall m$)
is $\exists$-free.  If we like, we could write it as $\forall f (0=0 \to \exists x (\dots))$ to coincide with the
$\forall x (p_1 \to \exists y\,p_2 )$ problem format.  Using this formulation for ${\sf P} \text{:}\rt(2,2)$ and
${\sf Q}\text{:}\rt(2,4)$, the predicate $R$ as in the statement of Theorem~\ref{B4} is in $\Gamma_1$.

Consider the following well-known proof of $\rt (2,4)$ from two applications of $\rt (2,2)$.
Given $f\colon [\nat]^2 \to 4$, define $g_1\colon [\nat ]^2 \to 2$ by setting $g_1(n,m) = 1$ if $f(n,m)>1$ and $g_1(n,m)=0$ otherwise.
Applying $\rt(2,2)$, let $x=\{x_0 , x_1, \dots\}$ be an infinite monochromatic set for $g_1$.  Note that $f([x]^2)$ is
either contained in $\{0,1\}$ or contained in $\{2,3\}$.  Define $g_2\colon [\nat]^2 \to 2$ by $g_2(n,m) = 1$ if $f(x_n , x_m ) $ is odd
and $g_2(n,m) = 0$ otherwise.  Applying $\rt (2,2)$ a second time,
let $y$ be an infinite monochromatic set for $g_2$.  Then $z=\{x_m \mid m\in y\}$ is
an infinite monochromatic set for $f$, completing the proof of $\rt (2,4)$.  This proof that
$\rt(2,2)$ implies $\rt(2,4)$ can be carried out in $i\rcaw$.  However, our work from previous sections shows that
the second use of $\rt(2,2)$ cannot be eliminated.

\begin{theorem}\label{C1}
$i\rcaw$ cannot prove $\rt(2,4)$ with one typical use of $\rt (2,2)$.
\end{theorem}

\begin{proof}
As noted in the second paragraph of this section, for our formulation of $\rt(2,2)$ and $\rt(2,4)$,
$\sf P$, $\sf Q$, and $R$ satisfy the hypotheses of Theorem~\ref{B4}.
By Corollary~3.4 of Dorais {\it et al.}~\cite{doraisetal}, $\rt(2,4) \not\wlt \rt(2,2)$.  By Corollary~\ref{4F},
$i\rcaw \not\vdash \rt(2,4) \wlt \rt(2,2)$.  
By Corollary~\ref{B5}, $i\rcaw$ does not prove $\rt(2,4)$ with one
typical use of $\rt(2,2)$.
\end{proof}

Theorem~3.3 of Hirschfeldt and Jockusch~\cite{hj} asserts that if $j,k,n \in \omega$ satisfy the inequalities $n\ge 1$ and $k>j\ge 2$,
then $\rt(n,k) \not\wlt \rt(n,j)$.  They note that this result was also proved independently by Patey and by Brattka and Rakotoniaina.
Substituting this result for the use of Corollary~3.4 in the proof of Theorem~\ref{C1} yields the following extension.

\begin{corollary}\label{C1coro}
If $j,k,n\in \omega$ satisfy $n \ge 1$ and $k>j\ge 2$, then $i\rcaw$ cannot prove $\rt(n,k)$ with one typical use of $\rt(n,j)$.
\end{corollary}

Returning to our original discussion of Ramsey's theorem pairs, we next show that $\rt(2,4)$ can be proved with
one typical use of $\rt(2,2)$ in systems such as $\rca$ that include the law of the excluded middle. This somewhat counterintuitive result relies on the following definition. 

\begin{defn}\label{C2}
$(\rca)$  Suppose $f\colon [\nat]^2 \to 4$.  A set $x$ is {\em $2$-mono} for $f$ if there is a set $\{i,j\} \subset \{0,1,2,3\}$ such that
$f([x]^2) \subset \{i,j\}$.
\end{defn}

\begin{theorem}\label{C3}
$\rca$ can prove $\rt(2,4)$ with one typical use of $\rt(2,2)$.
\end{theorem}

\begin{proof}
The following proof can be can be carried out in $\rca$.

Suppose $f\colon [\nat]^2 \to 4$.  Either there is an infinite $x\subset \nat$ that is $2$-mono for $f$
or there is no such set.  If there is such a set, define $j=0$, let $x$ be an increasing enumeration of such a set,
and suppose $f([x]^2) \subset \{a_0 , a_1\}$.  If there is no such set, define $j=1$ and let $x$ be an increasing
enumeration of~$\nat$.  Define $g\colon [\nat]^2 \to 2$ by the following:
\[
g(m,n) =
\begin{cases}
0 &\text{if~}j=0\text{~and~}f(x(m),x(n))= a_0\\
1 &\text{if~}j=0\text{~and~}f(x(m),x(n))= a_1\\
0 &\text{if~}j=1\text{~and~}f(x(m),x(n))\le 1\\
1 &\text{if~}j=1\text{~and~}f(x(m),x(n))\ge 2.
\end{cases}
\]
By one typical application of $\rt(2,2)$, let $y$ be an infinite monochromatic set for $g$.
If $j = 1$, then $y$ is an infinite $2$-mono set for $f$, contradicting the definition of $j$.
Thus $j=0$, and the set $z=\{x(m) \mid m\in y\}$ is an infinite monochromatic set for $f$.
\end{proof}

Using similar but more complicated constructions, for each standard integer $k$ one can show that
$\rt(2,k)$ can be proved with a single application of $\rt(2,2)$ in the classical system $\rca$.  For
example, given $f\colon [\nat]^2 \to 8$ either $\nat$ contains no infinite $4$-mono set, or there is an infinite
$4$-mono set with no infinite $2$-mono subset, or there is an infinite $2$-mono set.  Define
$g$ based on these possibilities and proceed as above.  Furthermore, examination of the proof of
Theorem~\ref{C3} reveals no actual use of the exponent.  Consequently, we can extend Theorem~\ref{C3}
as follows.

\begin{corollary}\label{C4}
Let $n$ and $k$ be positive elements of $\omega$.  $\rca$ can prove $\rt (n,k)$ with one typical use of $\rt (n,2)$.
\end{corollary}

This kind of nonconstructive argument leveraging a single typical use of an axiom is not limited to Ramsey's theorem. For example, for each $n \in \omega$, $\rca$ can prove ``every set has an $n$th Turing jump'' with a single typical use of ``every set has a Turing jump''. Many more examples come to mind, where a single typical use can be used to iterate a principle any finite number of times.

The relationship between Weihrauch reducibility and proofs in intuitionistic
systems played an important role in obtaining the results of this section.
We did not discover the
proof described in Theorem~\ref{C3} until our work on Corollary~\ref{B5}
indicated
the significance
of the law of the excluded middle in this setting.

\bibliographystyle{amsplain}

\newpage

\centerline{\Large Corrigendum: Using Ramsey’s theorem once}
\vskip .2in

\centerline{\large Jeffry L.~Hirst \qquad Carl Mummert}
\vskip .2in

\centerline{\large June 23, 2020}
\vskip .5in

In \cite{rt1}, the first definition should be as follows:

\setcounter{theorem}{0}
\begin{defn}
Suppose $\gth$ is a theory
extending intuitionistic predicate calculus
and $\probP$ and $\probQu$ are problems.
We say $\gth$ {\sl proves} $\sf Q$ {\sl with one typical use of} $\sf P$ if the following two sentences hold:
\begin{enumlist}
\item \label{B1A} For a variable $u$ there is a term $x_u$ such that using only axioms of $\gth$ and the assumption $q_1 (u)$,
and
holding the free variables of $q_1(u)$ constant,
there is a deduction
of $p_1(x_u)$.
\item \label{B1B}  For a previously unused variable $y$, there is a term $v_{x_u, y}$ such that using only axioms of $\gth$,
lines from the proof in sentence \ref{B1A}, and the assumptions $q_1(u)$ and $p_2(x_u , y)$,
while
holding the free variables of  $q_1(u)$ and $p_2(x_u , y)$ constant,
there is a deduction of $q_2 (u, v_{x_u , y} )$.
\end{enumlist}
\end{defn}

The revised definition applies to theories extending intuitionistic predicate calculus, matching the
formulation of Lemma 1, which immediately follows the definition in the article.  The restrictions on holding
variables constant are exactly those needed for the applications of the deduction theorem in the proof of
Lemma 1.  Essentially, Definition 1 divides a proof into two parts, before and after a single application
of $\sf P$.  The second portion of the proof may make use of the lines from the first portion
as noted in the second sentence of the revised definition.  This modification is useful in the proof
of Theorem 4, the last theorem of the article.

\end{document}